\documentclass[preprint,12pt]{elsarticle}

\usepackage{amsthm}
\usepackage{amsmath,amssymb,txfonts}
\usepackage{amsthm,color,bm,comment}

\newtheorem{theorem}{Theorem}
\newtheorem{lemma}{Lemma}
\newtheorem{prop}{Proposition}

\newtheorem{remark}{Remark}

\theoremstyle{definition}

\def\re{\mathbb{R}}
\def\N{\mathbb{N}}

\def\({\left(}
\def\){\right)}
\def\[{\left[}
\def\]{\right]}
\def\pd{\partial}

\def\ep{\varepsilon}
\def\w{\omega}
\def\la{\lambda}

\def\intO{\int_{\Omega}}

\newtheorem{ThmA}{Theorem A}

\begin{document}

\begin{frontmatter}



\title{Extremal functions of generalized critical Hardy inequalities}


\author[M]{Megumi Sano}
\ead{megumisano0609@gmail.com}
\address[M]{Department of Mathematics, Tokyo Institute of Technology, O-okayama, Meguro-ku, Tokyo 152-8551, Japan}


\date{\today}

\begin{abstract}
In this paper, we show the existence and non-existence of minimizers of the following minimization problems which include an open problem mentioned by Horiuchi and Kumlin \cite{HK}:
\begin{align*}
G_a := \inf_{u \in W_0^{1,N}(\Omega ) \setminus \{ 0\} } \dfrac{\int_{\Omega} |\nabla u |^{N} \,dx}{\( \int_{\Omega} |u|^{q} f_{a,\beta}(x) dx \)^{\frac{N}{q}}}, \,\text{where} \,\,f_{a, \beta}(x):=|x|^{-N}\( \log \frac{aR}{|x|} \)^{-\beta}.
\end{align*}
First, we give an answer to the open problem when $\Omega =B_R(0)$.
Next, we investigate the minimization problems on general bounded domains. 
In this case, the results depend on the shape of the domain $\Omega$. 
Finally, symmetry breaking property of the minimizers is proved for sufficiently large $\beta$.
\end{abstract}

\end{frontmatter}



%
%
\section{Introduction}

Let $N \ge 2$, $\Omega$ be a bounded domain in $\re^N$, $0 \in \Omega$, and $1<p<N$. The classical Hardy inequality holds for all $u \in W^{1,p}_0(\Omega)$ as follows:
\begin{equation}
\label{H_p}
\( \frac{N-p}{p} \)^p \intO \frac{|u|^p}{|x|^p} dx \le \intO | \nabla u |^p dx,
\end{equation}
where $W_0^{1,p}(\Omega)$ is a completion of $C_c^{\infty}(\Omega)$ with respect to the norm $\| \nabla (\cdot )\|_{L^p(\Omega)}$. We refer the celebrated work by G. H. Hardy \cite{Hardy}.
The inequality (\ref{H_p}) has great applications to partial differential equations, for example stability, global existence, and instantaneous blow-up and so on. See e.g. \cite{BV}, \cite{BG}. 
It is well-known that in (\ref{H_p}) $(\frac{N-p}{p})^p$ is the optimal constant and is not attained in $W_0^{1,p}(\Omega)$.


On the other hand, in the critical case where $p=N$, the following inequality which is called the critical Hardy inequality holds for all $u \in W^{1,N}_0(\Omega)$ and all $a \ge 1$, where $R =\sup_{x \in \Omega} |x|$\,:
\begin{align}
\label{H_N}
\( \frac{N-1}{N} \)^N \intO \dfrac{|u|^N}{|x|^N (\log \frac{aR}{|x|})^N} dx \le \intO | \nabla u |^N dx.
\end{align}
See e.g. \cite{Leray}, \cite{Ladyzhenskaya}, \cite{BFT(IUMJ)}, \cite{BFT(TAMS)}, \cite{GM book} Corollary 9.1.2., \cite{MOW}, \cite{TF}.
It is known that in (\ref{H_N}) $(\frac{N-1}{N})^N$ is the optimal constant and is not attained for any bounded domain $\Omega$ with $0 \in \Omega$ (see \cite{AS}, \cite{AE}, \cite{II}, \cite{BT} etc.).

In this paper, we consider optimal constants and its attainability of the following inequalities (\ref{GH_N}) which are generalizations of (\ref{H_N}):
\begin{align}\label{GH_N}
G_a \( \intO \frac{|u|^q}{ |x|^N (\log \frac{aR}{|x|} )^{\beta}} dx \)^{\frac{N}{q}} \le \intO |\nabla u |^N dx
\end{align}
for $u \in W_0^{1,N}(\Omega ), q,\beta >1$, and $a \ge 1$.
We define $G_a$ and $G_{a, {\rm rad}}$ as the optimal constants of the inequalities (\ref{GH_N}) as follows:
\begin{align}
\label{GH_N const}
	G_a := \inf_{u \in W_0^{1,N}(\Omega ) \setminus \{ 0\} } \dfrac{\intO | \nabla u |^N \,dx}{\( \intO \frac{|u|^q}{|x|^N (\log \frac{aR}{|x|})^{\beta}} dx \)^{\frac{N}{q}}}, \,\, G_{a,{\rm rad}} := \inf_{u \in W_{0, {\rm rad}}^{1,N}(\Omega ) \setminus \{ 0\} } \dfrac{\intO | \nabla u |^N \,dx}{\( \intO \frac{|u|^q}{|x|^N (\log \frac{aR}{|x|})^{\beta}} dx \)^{\frac{N}{q}}},
\end{align}
where $W_{0,{\rm rad}}^{1,N}(\Omega)= \{ \, u \in W_0^{1,N}(\Omega) \, | \, u \,\,\text{is radial} \, \}$.
When $\Omega = B_R(0), \, \beta =\frac{N-1}{N} q +1$, and $q > N$, the exact optimal constant and the attainability of $G_{a, {\rm rad}}$ are investigated by Horiuchi and Kumlin \cite{HK}. However we do not know the attainability of $G_a$ even if $\Omega = B_R(0)$. In fact, in their article \cite{HK} they mention that the attainability of $G_a$ is an open problem. See also \cite{H}. Note that the continuous embedding $W_0^{1,N}(B_R(0)) \hookrightarrow L^q(B_R(0) ; |x|^{-N}(\log \frac{aR}{|x|})^{-\beta}dx)$ is not compact when $\beta =\frac{N-1}{N}q +1, q \ge N$, and $a >1$. In addition, the rearrangement technique does not work due to the lack of monotone decreasing property of the potential function $|x|^{-N}(\log \frac{aR}{|x|})^{-\beta}$ when $1\le a < e^{\frac{\beta}{N}}$.

In this paper, we study the existence, non-existence, and symmetry breaking property of the minimizers of $G_a$.
First, we give an answer to the open problem except for $a=a_*$ which is a threshold number when $\Omega = B_R(0)$. More precisely, we show that there exists a minimizer of $G_a$ for $a \in (1, a_*)$ and there is no minimizer for $a > a_*$. 
Next, we extend the results to general bounded domains. Furthermore we investigate the positivity and the attainability of $G_1$ in general bounded domains. When $a=1$, the positivity and the attainability of $G_1$ depend on geometry of the boundary of the domain since the potential function has singularities on the boundary. 
Finally, we show that when $\Omega =B_R(0)$, any minimizers of $G_a$ are non-radial for large $\beta$ and fixed $q > N$, and any minimizers are radial for any $\beta$ and any $q \le N$.

Our problem is regarded as the critical case of one of Caffarelli-Kohn-Nirenberg type inequalities, see \cite{HK}. 
In the weighted subcritical Sobolev spaces $W_0^{1,p}(|x|^{\alpha} dx)$ where $p < N +\alpha$, the existence, nonexistence, and symmetry breaking property of the minimizers of Caffarelli-Kohn-Nirenberg type inequalities are well-studied especially for $p=2$, see \cite{T}, \cite{L}, \cite{CC}, \cite{Horiuchi}, \cite{CM}, \cite{CW1}, \cite{CW2}, \cite{SSW}, \cite{GK}, \cite{GR(2006)}, \cite{CL} and references therein.

Our minimization problem (\ref{GH_N const}) is related to the following nonlinear elliptic equation with the singular potential:
\begin{align}\label{EL}
\begin{cases}
-\text{div} \,( \, |\nabla u|^{N-2} \nabla u \,) = b \frac{|u|^{q-2}u}{|x|^N (\log \frac{aR}{|x|})^{\beta}} \quad &\text{in} \,\, \Omega, \\
\qquad u = 0  &\text{on} \,\, \pd \Omega.
\end{cases}
\end{align}
The minimizer for $G_a$ is a ground state solution of the Euler-Lagrange equation (\ref{EL}) with a Lagrange multiplier $b$. 

This paper is organized as follows:
In section \ref{Pre}, necessary preliminary facts are presented. 
In section \ref{ball case}, we prove the (non-)attainability of $G_a$ when $\Omega =B_R(0)$ and $a>1$. 
In section \ref{bdd case}, we extend the results to several bounded domains, and we investigate the positivity and the attainability of $G_1$ in several bounded domains.
In section \ref{sym break}, we show that symmetry breaking phenomena of the minimizers of $G_a$ occur for large $\beta$.

We fix several notations: 
$B_R(0)$ and $B_R^N(0)$ denote a $N$-dimensional ball centered $0$ with radius $R$ and $\omega_{N-1}$ denotes an area of the unit sphere $\mathbb{S}^{N-1}$ in $\re^N$. $|A|$ denotes the Lebesgue measure of a set $A \subset \re^N$. 
The Schwarz symmetrization $u^{\#} \colon \re^N \to [0, \infty]$ of $u$ is given by 
\begin{align*}
u^{\#}(x) = u^{\#}(|x|)=\inf \left\{ \tau >0 \,: \, | \{ y \in \re^N \, :\, |u(y)| > \tau \} \,| \le |B_{|x|}(0) | \right\}.
\end{align*}

%
%

\section{Preliminaries}\label{Pre}

In this section, we give a necessary and sufficient condition of the positivity of $G_a$ for $a \in [1, \infty)$. Furthermore we give the explicit value of $G_{a,{\rm}}$ and the minimizers when $\beta = \frac{N-1}{N}q +1$ and $q >N$. 
First, we give a necessary and sufficient condition (\ref{positivity cond}) of the positivity of $G_a$ when $a>1$.

\begin{prop}\label{GH_N positive}
Let $a>1$, $\Omega \subset \re^N$ be a bounded domain with $0 \in \Omega$, $R=\sup_{x \in \Omega} |x|, N \ge 2$ and $q, \beta>1$. Then $G_a > 0$ if and only if $\beta$ and $q$ satisfy
\begin{align}\label{positivity cond}
\text{either}\,\,\,\beta > \frac{N-1}{N} q + 1 \,\,\,\text{or}\,\,\,\beta = \frac{N-1}{N} q + 1, q \ge N.
\end{align}
\end{prop}

Essentially, Proposition \ref{GH_N positive} is proved by the following theorem in \cite{MOW(ref)}.
The authors in \cite{MOW(ref)} show a necessary and sufficient condition of the positivity for more general inequalities in the critical Sobolev-Lorentz spaces $H^s_{p,q}(\re^N)$. 
Note that the norm of $H^1_{N,N}(\re^N)$ is equivalent to it of $W^{1,N}(\re^N)$. 
We can obtain Proposition \ref{GH_N positive} from Theorem A and simple calculations. We omit the proof here.

\begin{ThmA}(\cite{MOW(ref)} Theorem 1.1.)
Let $N \in \N, 1<p<\infty, 1<r \le \infty$ and $1<\alpha, \beta <\infty$. Then there exists a constant $C>0$ such that for all $u \in H^{\frac{N}{p}}_{p,r}(\re^N)$, the inequality
\begin{align}\label{MOW}
\( \int_{B_{\frac{1}{2}}(0)} \frac{|u|^\alpha}{ |x|^N (\log \frac{1}{|x|} )^{\beta}} dx \)^{\frac{1}{\alpha }} \le C \| u \|_{H^{\frac{N}{p}}_{p,r}(\re^N)}
\end{align}
holds true if and only if one of the following conditions (i)`(iii) is fulfilled
\begin{align}\label{MOW condition}
\begin{cases}
(i) \,&1+\alpha -\beta < 0, \\
(ii) \,&1+\alpha -\beta \ge 0 \,\, \text{and} \,\, r<\frac{\alpha}{1+\alpha -\beta}, \\
(iii) \,&1+\alpha -\beta > 0, r = \frac{\alpha}{1+\alpha -\beta}, \,\,\text{and}\,\, \alpha \ge \beta. 
\end{cases}
\end{align}
\end{ThmA}

Next, we give a necessary and sufficient condition of the positivity of $G_a$ when $a=1$ and $\Omega =B_R(0)$.
Essentially, the following proposition follows from results in \cite{HK}. 

\begin{prop}\label{G_1 positive}
Let $\Omega = B_R(0)$. Then $G_1 > 0$ if and only if $\beta=q=N$.
\end{prop}

In \S 4, we extend Proposition \ref{G_1 positive} to general bounded domains. Proposition \ref{G_1 positive} follows from Proposition \ref{G_1 positive general} in \S 4. Thus we omit the proof of Proposition \ref{G_1 positive} here.

Finally, we give the explicit value of the optimal constant $G_{a, {\rm rad}}$ and the minimizers when $\beta =\frac{N-1}{N}q +1$ and $q>N$. 

Logarithmic transformations related to $G_{a, {\rm rad}}$ are founded by \cite{HK}, \cite{II 1}, \cite{Z}, \cite{ST}. 
Especially, in the radial setting, the authors in \cite{ST} show an unexpected relation (\ref{H_p to H_N eq}) that the critical Hardy inequality in dimension $N \ge 2$ is equivalent to the one of the subcritical Hardy inequalities in higher dimension $m > N$ by using a transformation (\ref{trans}) as follows:
\begin{align}
\label{H_p to H_N eq}
	&\int_{\re^m} | \nabla u |^N \,dx - \( \frac{m-N}{N} \)^N \int_{\re^m} \frac{|u|^N}{|x|^N} \,dx \notag \\
	&= \frac{\w_{m-1}}{\w_{N-1}} \( \frac{m-N}{N-1} \)^{N-1} \( \int_{B_R^N(0)} | \nabla w |^N \,dy - \( \frac{N-1}{N} \)^N \int_{B_R^N(0)} \dfrac{|w|^N}{|y|^N \( \log \frac{R}{|y|} \)^N} \,dy \),  \\
\label{trans}
&\text{where} \,\,u(|x|) = w(|y|)\,\,\text{and}\,\, \( \log \frac{R}{|y|} \)^{\frac{N-1}{N}} = |x|^{-\frac{m-N}{N}}.
\end{align}
By using the transformation (\ref{trans}) and direct calculations, we can observe not only an equivalence between two Hardy inequalities but also the equivalence between Hardy-Sobolev type inequalities and generalized critical Hardy inequalities in the radial setting as follows:
\begin{align}\label{G_rad to HS}
G_{1, {\rm rad}} &= \inf_{w \in W_{0, {\rm rad}}^{1,N}(B_R^N(0)) \setminus \{ 0\}} 
\dfrac{\int_{B_R^N(0)} | \nabla w |^N \,dy}{\( \int_{B_R^N(0)} \frac{|w|^q}{|y|^N (\log \frac{R}{|y|})^{\beta}} dy \)^{\frac{N}{q}}} \notag \\
&= \( \frac{\w_{N-1}}{\w_{m-1}} \)^{1-\frac{N}{q}} \( \frac{N-1}{m-N} \)^{N-1+\frac{N}{q}} \inf_{u \in W_{0, {\rm rad}}^{1,N}(\re^m ) \setminus \{ 0\}} \dfrac{\int_{\re^m} | \nabla u |^N \,dx}{\( \int_{\re^m} |x|^{\alpha}|u|^q dx \)^{\frac{N}{q}}},
\end{align}
where $\alpha =\frac{m-N}{N-1}(\beta -1) -m$. The authors in \cite{ST} also give a transformation which is a modification of (\ref{trans}) when $a>1$. Since the minimization problems on the right hand side of (\ref{G_rad to HS}) are well-known (see e.g. \cite{T}, \cite{L}), we can obtain the following proposition by using these transformations. 

\begin{prop}\label{G_rad}
Let $\beta =\frac{N-1}{N}q +1, q > N$, and $\Omega =B_R(0)$. Then the followings hold.

\noindent
(i) $G_{a,{\rm rad}}$ is independent of $a \ge 1$. Furthermore, the exact value of the optimal constant is as follows:
\begin{align*}
G_{a, {\rm rad}} =G_{\rm rad} := \w_{N-1}^{1-\frac{N}{q}} (N-1) \( \frac{N}{q} \)^{1-\frac{2N}{q}} \( 1-\frac{N}{q} \)^{-2+\frac{2N}{q}} \( \frac{\Gamma \( \frac{q(N-1)}{q-N} \) \Gamma \( \frac{N}{q-N} \) }{\Gamma \( \frac{qN}{q-N} \)}  \)^{1-\frac{N}{q}},
\end{align*}
where $\Gamma(\cdot)$ is the gamma function. 

\noindent
(ii) $G_{a, {\rm rad}}$ is not attained for any $a >1$. 

\noindent
(iii) $G_{1, {\rm rad}}$ is attained by the family of the following functions $U_\la$:
\begin{align*}
U_\la (y)= C \la^{-\frac{N-1}{N}} \( 1+ \( \la \log \frac{R}{|y|} \)^{-\frac{q-N}{N}} \)^{-\frac{N}{q-N}}, \,\text{where}\,\, C \in \re \setminus \{ 0\} \,\,\text{and}\,\,\la >0.
\end{align*}
\end{prop}

Here, we give a simple proof of Proposition \ref{G_rad} (ii) by using a scaling argument.

\begin{proof}[{\bf Proof of Proposition \ref{G_rad} (ii)}]
Let $\beta =\frac{N-1}{N}q +1, q>N$, and $a >1$. Assume that $u \in W_{0,\text{rad}}^{1,N}(B_R(0))$ is a radial minimizer of $G_{a, {\rm rad}}$. We can assume that $u$ is nonnegative without loss of generality. 
We shall derive a contradiction. 
For $\la \in (0, 1)$, we consider a scaled function $u_\la \in W_{0,\text{rad}}^{1,N}(B_{R}(0))$ which is given by 
\begin{align*}
u_\la (x) = 
\begin{cases}
\la^{-\frac{N-1}{N}} u\( \( \frac{|x|}{aR} \)^{\la -1} x \) \quad &\text{if}\,\, x \in B_{a^{(1-\la^{-1} )} R}(0), \\
0 &\text{if}\,\, x \in B_R(0) \setminus B_{a^{(1-\la^{-1} )}R}(0).
\end{cases}
\end{align*}
Then we have
\begin{align*}
\dfrac{\int_{B_R(0)} | \nabla u |^N \,dx}{\( \int_{B_R(0)} \frac{|u|^q}{|x|^N (\log \frac{aR}{|x|})^{\beta}} dx \)^{\frac{N}{q}}} = \dfrac{\int_{B_R (0)} | \nabla u_\la |^N \,dx}{\( \int_{B_R (0)} \frac{|u_\la |^q}{|x|^N (\log \frac{aR}{|x|})^{\beta}} dx \)^{\frac{N}{q}}}
\end{align*}
which yields that $u_\la$ is also a nonnegative minimizer of $G_{a, {\rm rad}}$. 
On the other hand, we can show that $u_\la \in C^1(B_R(0) \setminus \{ 0\} )$ and $u_\la >0$ in $B_R(0) \setminus \{ 0\}$ by standard regularity argument and strong maximum principle to the Euler-Lagrange equation (\ref{EL}), see e.g. \cite{D}, \cite{PS}. 
However $u_\la \equiv 0$ in $B_R(0) \setminus B_{a^{(1-\la^{-1} )}R}(0)$. 
This is a contradiction. Therefore $G_{a, {\rm rad}}$ is not attained.
\end{proof}


\section{Existence and non-existence of the minimizers}\label{ball case}

Let $\Omega =B_R(0)$. 
In this section, we prove an existence and non-existence of the minimizers of $G_a$. 
First result is as follows.

\begin{theorem}\label{Thm ball}
Let $a>1$ and $q, \beta >1$ satisfy (\ref{positivity cond}). Then the followings hold.

\noindent
(i) If \,$\beta > \frac{N-1}{N}q +1$, then $G_a$ is attained.

\noindent
(ii) If $\beta = \frac{N-1}{N}q +1$ and $q > N$, then there exists $a_* \in (1, e^{\frac{\beta}{N}}]$ such that $G_a$ is attained for $a \in (1,a_*)$ and $G_a$ is not attained for $a > a_*$.
\end{theorem}

\begin{remark}
If $a_* = e^{\frac{\beta}{N}}$, then we can show that $G_{a_*}$ is not attained. In fact, if we assume that $G_{a_*}$ is attained by $u$, then $u^{\#}$ is a radial minimizer of $G_{a,{\rm rad}}$ which contradicts Proposition \ref{G_rad} (ii), see the proof of Theorem \ref{Thm ball} (ii). 
However we do not know the value of $a_*$.
\end{remark}

In order to show Theorem \ref{Thm ball}, we need three lemmas. 
First we show the (non-)compactness of the embedding $W_0^{1,N}(B_R(0)) \hookrightarrow L^q(B_R(0); f_{a, \, \beta}(x) dx)$, where $f_{a, \, \beta}(x)=|x|^{-N} \( \log \frac{aR}{|x|} \)^{-\beta}$. 

\begin{lemma}\label{emb}
Let $a>1$ and $q, \beta >1$ satisfy (\ref{positivity cond}). Then the continuous embedding $W_0^{1,N}(B_R(0)) \hookrightarrow L^q(B_R(0); f_{a, \, \beta}(x) dx)$ is 

\noindent
(i) compact if \,$\beta > \frac{N-1}{N}q +1$,

\noindent
(ii) non-compact if \,$\beta = \frac{N-1}{N}q +1$ and $q \ge N$.
\end{lemma}

\begin{proof}[{\bf Proof of Lemma \ref{emb}}]
\noindent
(i) It is proved in \cite{ST(EJDE)}. However we give a proof here for the convenience of readers. 
Let $(u_m)_{m=1}^{\infty} \subset W_0^{1,N}(B_R(0))$ be a bounded sequence. Then there exists a subsequence $(u_{m_k})_{k=1}^{\infty}$ such that
\begin{align}\label{L^r}
&u_{m_k} \rightharpoonup u \,\, \text{in} \,\, W_0^{1,N}(B_R(0) ), \nonumber \\
&u_{m_k} \to u \,\, \text{in} \,\, L^r(B_R(0) ) \quad \text{for any} \,\, r \in [1, \infty ).
\end{align}
Let $\alpha$ satisfy $\frac{N-1}{N}q +1 < \alpha < \beta$. For all $\ep >0$, there exists $\delta >0$ such that
\begin{equation}\label{log}
\( \log \frac{aR}{|x|} \)^{\alpha -\beta} < \ep \quad \text{for all} \,\, x \in B_\delta (0).
\end{equation}
From (\ref{L^r}) and (\ref{log}), we have
\begin{align*}
\int_{B_R(0)} \frac{|u_{m_k}-u|^q}{ |x|^N (\log \frac{aR}{|x|} )^{\beta}} dx
&\le \ep \int_{B_\delta (0)} \frac{|u_{m_k}-u|^q}{ |x|^N ( \log \frac{aR}{|x|} )^{\alpha}} dx + \delta^{-N} \( \log \frac{aR}{\delta}\)^{-\beta} \| u_{m_k} -u \|^q_{L^q(B_R(0) )} \\
&\le \ep  C  \| \nabla (u_{m_k} -u) \|^q_{L^N(B_R(0) )} + C \| u_{m_k} -u \|^q_{L^q(B_R(0) )} \\
&\le C \ep + C \| u_{m_k} -u \|^q_{L^q(B_R(0) )} \to 0 \quad \text{as} \,\, \ep \to 0, k \to \infty.
\end{align*}
Thus the continuous embedding $W_0^{1,N}(B_R(0)) \hookrightarrow L^q(B_R(0); f_{a, \, \beta}(x) dx)$ is compact if \,$\beta > \frac{N-1}{N}q +1$.

\noindent
(ii) We can see a non-compact sequence $(u_{\frac{1}{m}})_{m=1}^\infty$ in $W_0^{1,N}(B_R(0))$, where for $\la \in (0, 1]$ $u_\la$ is defined in the proof of Proposition \ref{G_rad} (ii). Hence the continuous embedding $W_0^{1,N}(B_R(0)) \hookrightarrow L^q(B_R(0); f_{a, \, \beta}(x) dx)$ is non-compact if $\beta =\frac{N-1}{N}q +1$ and $q \ge N$.
\end{proof}

In \cite{HK}, a continuity of $G_a$ with respect to $a$ is proved for $a \in (1, \infty)$. However, in our argument, the continuity of $G_a$ at $a=1$ is needed. 

\begin{lemma}\label{conti G_a}
$G_a$ is monotone increasing and continuous with respect to $a \in [1, \infty)$.
\end{lemma}

\begin{proof}[Proof of Lemma \ref{conti G_a}]
It is enough to show only the continuity of $G_a$ at $a=1$. From the definition of $G_1$, we can take $(u_m)_{m=1}^\infty \subset C_c^{\infty}(B_R(0))$ and $R_m < R$ for any $m$ such that supp $u_m \subset B_{R_m}(0), R_m \nearrow R$, and
\begin{align*}
\dfrac{\int_{B_{R_m}(0)} | \nabla u_m |^N \,dx}{\( \int_{B_{R_m}(0)} |u_m|^q f_{1,\,\beta}(x)\, dx \)^{\frac{N}{q}}} = G_1 + o(1) \quad {\rm as }\,\, m \to \infty.
\end{align*}
Set $v(y)=u_m(x)$, where $y=\frac{R}{R_m}x$. Then 
\begin{align*}
\dfrac{\int_{B_{R_m}(0)} | \nabla u_m |^N \,dx}{\( \int_{B_{R_m}(0)} |u_m|^q f_{1,\,\beta}(x)\, dx \)^{\frac{N}{q}}} 
= \dfrac{\int_{B_R(0)} | \nabla v |^N \,dx}{\( \int_{B_R(0)} |v|^q f_{a_m,\,\beta}(x)\, dx \)^{\frac{N}{q}}} \ge G_{a_m},
\end{align*}
where $a_m = \frac{R}{R_m} \searrow 1$ as $m \to \infty$. Therefore we have $G_{a_m} \le G_1 + o(1)$. Since $f_{a_m, \beta}(x) \le f_{1, \beta}(x)$ for any $x \in B_R(0)$, we have $G_1 \le G_{a_m}$. Hence we see that $\lim_{a \searrow 1} G_a =G_1$.
\end{proof}

Third Lemma is concerned with the concentration level of minimizing sequences of $G_a$.

\begin{lemma}\label{concentration level}
Let $\beta =\frac{N-1}{N}q +1, q>N$, and $a>1$. 
If $G_a < G_{\rm rad}$, then $G_a$ is attained, where $G_{\rm rad}$ is given by Proposition \ref{G_rad} (i).
\end{lemma}

It is easy to show Theorem \ref{Thm ball} by these three lemmas. 
Therefore we give a proof of Theorem \ref{Thm ball} before showing Lemma \ref{concentration level}.

\begin{proof}[Proof of Theorem \ref{Thm ball}]
(i) This is proved by Lemma \ref{emb} (i). We omit the proof here.

\noindent
(ii) Let $\beta =\frac{N-1}{N}q +1$ and $q>N$. When $a \ge e^{\frac{\beta}{N}}$, the potential function $f_{a, \, \beta}$ is radially decreasing. Thus the P\'olya-Szeg\"o inequality and the Hardy-Littlewood inequality imply that
\begin{align*}
\dfrac{\int_{B_R(0)} | \nabla u |^N \,dx}{\( \int_{B_R(0)} |u|^q f_{a,\,\beta}(x)\, dx \)^{\frac{N}{q}}} \ge \dfrac{\int_{B_R(0)} | \nabla u^{\#} |^N \,dx}{\( \int_{B_R(0)} |u^{\#}|^q f_{a,\,\beta}(x)\, dx \)^{\frac{N}{q}}} \ge G_{a, {\rm rad}}
\end{align*}
for any $u \in W_0^{1,N}(B_R(0))$ and $a \ge e^{\frac{\beta}{N}}$. Therefore $G_a = G_{a, {\rm rad}} =G_{\rm rad}$ for any $a \ge e^{\frac{\beta}{N}}$. Moreover we see $G_1 =0$ by Proposition \ref{G_1 positive}. Since $G_a$ is continuous and monotone increasing with respect to $a \in [1, \infty)$ by Lemma \ref{conti G_a}, there exists $a_* \in (1, e^{\frac{\beta}{N}}]$ such that $G_a < G_{\rm rad}$ for $a \in [1, a_*)$ and $G_a = G_{\rm rad}$ for $a \in [a_*, \infty)$. 
Hence $G_a$ is attained for $a \in (1, a_*)$ by Lemma \ref{concentration level}. 
On the other hand, if we assume that there exists a nonnegative minimizer $u$ of $G_a$ for $a > a_*$, then we can show that $u \in C^1(B_R(0) \setminus \{ 0\})$ and $u >0$ in $B_R(0) \setminus \{ 0\}$ by standard regularity argument and strong maximum principle to the Euler-Lagrange equation (\ref{EL}), see e.g. \cite{D}, \cite{PS}. 
Therefore we see that
\begin{align*}
G_{\rm rad} = G_a = \dfrac{\int_{B_R(0)} | \nabla u |^N \,dx}{\( \int_{B_R(0)} |u|^q f_{a,\,\beta}(x)\, dx \)^{\frac{N}{q}}} 
> \dfrac{\int_{B_R(0)} | \nabla u |^N \,dx}{\( \int_{B_R(0)} |u|^q f_{a_*,\,\beta}(x)\, dx \)^{\frac{N}{q}}} \ge G_{\rm rad}.
\end{align*}
This is a contradiction. Therefore $G_a$ is not attained for $a > a_*$.
\end{proof}

Finally, we prove Lemma \ref{concentration level}.

\begin{proof}[Proof of Lemma \ref{concentration level}]
Take a minimizing sequence $(u_m)_{m=1}^{\infty} \subset W_0^{1,N}(B_R(0))$ of $G_a$. Without loss of generality, we can assume that
\begin{align*}
\int_{B_R(0)} |u_m|^q f_{a,\beta}(x) \,dx =1,\quad \int_{B_R(0)} |\nabla u_m|^N \,dx = G_a + o(1) \,\,{\rm as}\,\, m \to \infty.
\end{align*}
Since $(u_m)$ is bounded in $W_0^{1,N}(B_R(0))$, passing to a subsequence if necessary, $u_m \rightharpoonup u$ in $W_0^{1,N}(B_R(0))$.
Then by Brezis-Lieb lemma, we have
\begin{align*}
G_a &= \int_{B_R(0)} |\nabla u_m|^N \,dx + o(1)\\
&= \int_{B_R(0)} |\nabla (u_m -u)|^N \,dx + \int_{B_R(0)} |\nabla u|^N \,dx +o(1) \\
&\ge G_a \( \int_{B_R(0)} |u_m-u|^q f_{a,\,\beta}(x)\, dx \)^{\frac{N}{q}} + G_a \( \int_{B_R(0)} |u|^q f_{a,\,\beta}(x)\, dx \)^{\frac{N}{q}} +o(1) \\
&\ge G_a \( \int_{B_R(0)} \( |u_m-u|^q + |u|^q \) f_{a,\,\beta} (x) \,dx \)^{\frac{N}{q}} +o(1) \\
&= G_a \( \int_{B_R(0)} |u_m|^q f_{a,\,\beta}(x)\, dx \)^{\frac{N}{q}} +o(1) =G_a
\end{align*}
which implies that either $u \equiv 0$ or $u_m \to u \nequiv 0$ in $L^q(B_R(0); f_{a, \, \beta}(x) dx)$ holds true from the equality condition of the last inequality. 
We shall show that $u \nequiv 0$. Assume that $u \equiv 0$. 
Then we claim that
\begin{align}\label{concentrate}
G_{\rm rad} \le \int_{B_R(0)} |\nabla u_m|^N \,dx +o(1).
\end{align}
If the claim (\ref{concentrate}) is true, then we see that $G_{\rm rad} \le G_a$ which contradicts the assumption. Therefore $u \nequiv 0$ which implies that $u_m \to u \nequiv 0$ in $L^q(B_R(0); f_{a, \, \beta}(x) dx)$. 
Hence we have 
\begin{align*}
1= \int_{B_R(0)} |u|^q f_{a,\,\beta}(x)\, dx, \quad \int_{B_R(0)} |\nabla u|^N \,dx \le \liminf_{m \to \infty} \int_{B_R(0)} |\nabla u_m|^N \,dx = G_a.
\end{align*}
Thus we can show that $u$ is a minimizer of $G_a$. 
We shall show the claim (\ref{concentrate}). 
Since $u_m \to 0$ in $L^r(B_R(0))$ for any $r \in [1, \infty)$ and the potential function $f_{a, \,\beta}$ is bounded away from the origin, for any small $\ep > 0$ we have
\begin{align*}
1= \int_{B_R(0)} |u_m|^q f_{a,\,\beta}(x)\, dx = \int_{B_{\frac{\ep R}{2}}(0)} |u_m|^q f_{a,\,\beta}(x)\, dx +o(1).
\end{align*}
Let $\phi_\ep$ be a smooth cut-off function which satisfies the followings:
\begin{align*}
0 \le \phi_\ep \le 1, \,\,\phi_\ep \equiv 1 \,\,{\rm on}\,\,B_{\frac{\ep R}{2}}(0),\,\,{\rm supp}\, \phi_\ep \subset B_{\ep R} (0), \,\,|\nabla \phi_\ep | \le C\ep^{-1}. 
\end{align*}
Set $\tilde{u_m}(y)=u_m(x)$ and $\tilde{\phi_\ep}(y)=\phi_\ep (x)$, where $y=\frac{x}{\ep}$. Then we have
\begin{align*}
1&= \( \int_{B_{\frac{\ep R}{2}}(0)} |u_m|^q f_{a,\,\beta}(x)\, dx \)^{\frac{N}{q}} +o(1) \\
&\le \( \int_{B_{\ep R} (0)} |u_m \phi_\ep |^q f_{a,\,\beta}(x)\, dx \)^{\frac{N}{q}} +o(1) \\
&=\( \int_{B_R (0)} |\tilde{u_m} \tilde{\phi_\ep} |^q f_{a\ep^{-1},\,\beta}(x)\, dx \)^{\frac{N}{q}} +o(1) 
\le G_{a\ep^{-1}}^{-1} \int_{B_R(0)} |\nabla (\tilde{u_m} \tilde{\phi_\ep } ) |^N \,dx +o(1).
\end{align*}
We see that $a \ep^{-1} \ge e^{\frac{\beta}{N}}$ for small $\ep$. Since $G_{a\ep^{-1}} =G_{a,\,{\rm rad}}=G_{\rm rad}$ from the proof of Theorem \ref{Thm ball} (ii), we have
\begin{align*}
1&\le G_{\rm rad}^{-1} \int_{B_R(0)} |\nabla (\tilde{u_m} \tilde{\phi_\ep } ) |^N \,dx +o(1) \\
&\le G_{\rm rad}^{-1} \( \int_{B_{\ep R}(0)} |\nabla u_m |^N \,dx + N \int_{B_{\ep R}(0) \setminus B_{\frac{\ep R}{2}}(0)} |\nabla u_m |^{N-2} \nabla u_m  \cdot \nabla \phi_\ep u_m \phi_\ep^{N-1} \,dx  \) +o(1) \\
&\le G_{\rm rad}^{-1} \( \int_{B_{\ep R}(0)} |\nabla u_m |^N \,dx + NC\ep^{-1} \| \nabla u_m \|_{L^N}^{N-1} \| u_m \|_{L^N} \) + o(1)\\
&\le G_{\rm rad}^{-1} \int_{B_{\ep R}(0)} |\nabla u_m |^N \,dx + o(1) 
\le G_{\rm rad}^{-1} \int_{B_R(0)} |\nabla u_m |^N \,dx +o(1).
\end{align*}
Therefore we obtain the claim (\ref{concentrate}). The proof of Lemma \ref{concentration level} is now complete.
\end{proof}

\section{In the case of general bounded domain}\label{bdd case}

We extend Theorem \ref{Thm ball} and Proposition \ref{G_1 positive} to bounded domains. 
Throughout this section, we assume that $\Omega \subset \re^N$ is a bounded domain, $0 \in \Omega$, and $\beta$ and $q$ satisfy (\ref{positivity cond}). 
Set $R=\sup_{x \in \Omega} |x|$. 

First 
we extend Proposition \ref{G_1 positive} to general bounded domains. 
If there exists $\Gamma \subset \pd \Omega \cap \pd B_R(0)$ such that $\Gamma$ is open in $\pd B_R(0)$, then we can obtain the same result as Proposition \ref{G_1 positive} as follows.

\begin{prop}\label{G_1 positive general}
Assume that there exists $\Gamma \subset \pd \Omega \cap \pd B_R(0)$ such that $\Gamma$ is open in $\pd B_R(0)$. 
Then $G_1 > 0$ if and only if $\beta=q=N$.
\end{prop}

\begin{proof}[Proof of Proposition \ref{G_1 positive general}]
First we show that $G_1 =0$ if $\beta > \frac{N-1}{N}q +1$. 
Set $x=r \w \, ( r=|x|, \w \in S^{N-1})$ for $x \in \re^N$. 
From the assumption, we can take $\delta >0$ and $\tilde{\Gamma} \subset \Gamma$ such that $\tilde{\Gamma}$ is open in $\pd B_R(0)$ and
\begin{align*}
\left\{ (r, \w) \in [0, R) \times S^{N-1} \,\Biggl| \, R-2\delta \le r \le R, \w \in \frac{1}{R} \tilde{\Gamma} \right\} \subset \Omega.
\end{align*}
Let $0 \nequiv \psi \in C_c^{\infty}(\frac{1}{R} \tilde{\Gamma})$ and $\phi \in C^{\infty}([0,\infty))$ satisfy $\phi \equiv 1$ on $[R-\delta, R]$ and $\phi \equiv 0$ on $[0,R-2\delta]$. Set $u_s(x) = \( \log \frac{R}{r} \)^{s} \psi ( \omega ) \phi (r)$. 
Then we have
\begin{align*}
\int_{\Omega} | \nabla u_s |^N dx &= \int_{S^{N-1}} \int_0^R \left| \frac{\pd u_s}{\pd r} \w + \frac{1}{r} \nabla_{S^{N-1}} u_s \right|^N r^{N-1} \, dr dS_{\w} \\
&\le 2^{N-1} \int_{S^{N-1}} \int_0^R \left| \frac{\pd u_s}{\pd r} \right|^N r^{N-1} + \left| \nabla_{S^{N-1}} u_s \right|^N r^{-1} \, dr dS_{\w} \\
&\le s^N C \int_{R-\delta}^R \( \log \frac{R}{r} \)^{(s-1)N} \frac{dr}{r} + C \int_{R-\delta}^R \( \log \frac{R}{r} \)^{sN} \frac{dr}{r} +C \\
&\le s^N C \int_0^{\log \frac{R}{R-\delta}} t^{(s-1)N} \,dt +C < \infty \quad \text{if}\,\,s > \frac{N-1}{N}.
\end{align*}
Thus $u_s \in W_0^{1,N}(\Omega )$ for all $s > \frac{N-1}{N}$.
However, direct calculation shows that
\begin{align*}
\intO \frac{|u_{s}|^q}{|x|^N \( \log \frac{R}{|x|} \)^{\beta} } dx 
\ge C \int_{R-\delta}^{R} \( \log \frac{R}{r} \)^{sq-\beta} \frac{dr}{r} =C \int_0^{\log \frac{R}{R-\delta}} t^{sq-\beta} \, dt
\end{align*}
which implies that
\begin{align*}
\intO \frac{|u_{s}|^q}{|x|^N \( \log \frac{R}{|x|} \)^{\beta} } dx = \infty
\end{align*}
for $s$ close to $\frac{N-1}{N}$ since $\beta > \frac{N-1}{N}q +1$.
Therefore we see that 
\begin{align}\label{beta>q}
G_1 =0 \quad \text{if} \quad \beta > \frac{N-1}{N}q +1.
\end{align}

Next we show that $G_1 =0$ if $\beta > N$.
Set $x_{\ep} = (R - 2\ep ) \frac{y}{R}$ for $y \in \pd B_R(0)$. Note that $B_{\ep}(x_{\ep}) \subset \Omega$ for small $\ep >0$ and some $y \in \Gamma$.
Then we define $u_\ep$ as follows:
\begin{align*}
u_{\ep}(x) =
	   \begin{cases}
		v\( \frac{|x-x_{\ep}|}{\ep} \) \,\,\,&\text{if} \,\,\, x \in B_{\ep}(x_{\ep}), \\
		0  &\text{if} \,\,\, x \in \Omega \setminus B_{\ep}(x_{\ep}),
	   \end{cases}
\,\, \text{where}\,\, v(t)=
	   \begin{cases}
		1 \,\,\,&\text{if} \,\,\,0\le t \le \frac{1}{2},  \\
		2(1-t)  &\text{if} \,\,\, \frac{1}{2} < t \le 1.
	   \end{cases}
\end{align*}
Since $\log t \le t-1$ for $t \ge 1$, we obtain
\begin{align*}
\intO | \nabla u_{\ep} (x)|^N \,dx &= \int_{B_1(0)} |\nabla v (|z|)|^N \,dz < \infty, \\
\intO \frac{|u_{\ep}(x)|^q}{|x|^N \( \log \frac{R}{|x|} \)^{\beta} } dx 
&\ge C \int_{B_\ep (x_\ep )} \frac{|u_{\ep}(x)|^q}{(R-|x|)^{\beta}} dx \ge  \frac{C}{(3\ep )^{\beta}} \int_{B_{\frac{\ep}{2}} (x_\ep )} dx =C \, \ep^{N-\beta} \to \infty 
\end{align*}
as $\ep \to 0$ when $\beta > N$.
Hence we see that 
\begin{align}\label{beta>N}
G_1 =0 \quad \text{if} \quad \beta > N.
\end{align}
From (\ref{beta>q}), (\ref{beta>N}), and (\ref{positivity cond}), we see that $G_1 >0$ if and only if $q=\beta =N$.
\end{proof}

If there does not exist $\Gamma$ in Proposition \ref{G_1 positive general}, then we can expect that the  relation between $q,\beta$ and the positivity of $G_1$ depends on geometry of the boundary $\pd \Omega$. 
In order to see it, we consider special cuspidal domains which satisfy the following conditions:
\begin{align*}
(\Omega_1): \,&\pd \Omega \cap \pd B_R(0) = \{ (0, \cdots , 0, -R) \}.\\
(\Omega_2): \,&\pd \Omega \,\text{is represented by a graph}\, \phi: \re^{N-1} \to [-R, \infty )\,\, \text{near the point}\,(0, \cdots, 0, -R). \\
&\text{Namely, for small $\delta >0$ the following holds true:}\\
&Q_\delta := \Omega \cap ( \re^{N-1} \times [-R, -R+\delta ] ) = \{ (x', x_N) \in \re^{N-1} \times [-R, -R+\delta ] \,|\,x_N > \phi (x') \}.\\
(\Omega_3): \,&\text{there exist}\, C_1, C_2 >0\, \text{and}\, \alpha \in (0, 1] \,\text{such that}\\
&C_1|x'|^{\alpha} \le \phi (x') +R \le C_2|x'|^{\alpha}\, \text{for any}\, x' \in \re^{N-1}.
\end{align*}

$\alpha$ in $(\Omega_3)$ expresses the sharpness of the cusp at the point $(0, \cdots, 0,-R)$.
Then we can obtain the following theorem concerned with the positivity and the attainability of $G_1$. 

\begin{theorem}\label{Thm cusp}
Assume that $\Omega$ satisfies the assumptions $(\Omega_1)-(\Omega_3)$. 
Then there exists $\beta^{*}=\beta^{*}(\alpha, q) \in [\frac{N-1}{\alpha} +1, \frac{N}{\alpha}]$ such that $G_1 =0$ for $\beta > \beta_*$ and $G_1 >0$ for $\beta < \beta^{*}$. 
Furthermore $G_1$ is attained for $\beta \in (\frac{N-1}{N}q+1, \beta^{*})$. 
\end{theorem}

\begin{remark}
When $\beta =q =N$ and $0 \in \Omega$, $G_1$ is not attained for any bounded domain. However, when $0 \not\in \Omega$, the attainability of $G_1$ depends on a geometry of the boundary $\pd \Omega$. Very recently, Byeon and Takahashi investigate the attainability of $G_1$ on cuspidal domains in their article \cite{BT} when $\beta =q =N$. 
\end{remark}

\begin{proof}[{\bf Proof of Theorem \ref{Thm cusp}}]
First we shall show that $G_1 =0$ if $\beta > \frac{N}{\alpha}$. 
From $(\Omega_3)$, we can observe that $B_{A \ep^{\frac{1}{\alpha}}}(x_{\ep}) \subset \Omega$ for small $\ep >0$ and small $A >0$, where $x_\ep =(0, \cdots, 0, -R+2\ep)$. 
Then we define $w_\ep$ as follows:
\begin{align*}
w_{\ep}(x) =
	   \begin{cases}
		v\( \frac{|x-x_{\ep}|}{A \ep^{\frac{1}{\alpha}}} \) \,\,\,&\text{if} \,\,\, x \in B_{A \ep^{\frac{1}{\alpha}}}(x_{\ep}), \\
		0  &\text{if} \,\,\, x \in \Omega \setminus B_{A \ep^{\frac{1}{\alpha}}}(x_{\ep}),
	   \end{cases}
\end{align*}
where $v$ is the same function in the proof of Proposition \ref{G_1 positive general}. 
In the same way as the proof of Proposition \ref{G_1 positive general}, we have
\begin{align*}
\intO | \nabla_x w_{\ep}(x) |^N \,dx < \infty, \,\, \intO \frac{|w_{\ep}(x)|^q}{|x|^N \( \log \frac{R}{|x|} \)^{\beta} } dx \ge C \, \ep^{\frac{N}{\alpha}-\beta} \to \infty 
\end{align*}
as $\ep \to 0$ if $\beta > \frac{N}{\alpha}$.
Therefore we have
\begin{align}\label{G=0}
G_1 =0 \quad \text{at least for}\,\, \beta > \frac{N}{\alpha}.
\end{align}

Next we shall show that $G_1 >0$ if $\beta < \frac{N-1}{\alpha}+1$. 
For $u \in W_0^{1,N}(\Omega)$, we divide the domain $\Omega$ into three parts as follows:
\begin{align}\label{123}
\intO \frac{|u(x)|^q}{|x|^N \( \log \frac{R}{|x|} \)^{\beta} } dx 
= \int_{\Omega \cap B_{\frac{R}{2}}(0)} + \int_{\Omega \setminus \( B_{\frac{R}{2}}(0) \cup Q_\delta \)} + \int_{Q_\delta } =: I_1 + I_2 + I_3.
\end{align}
From Theorem A, we obtain
\begin{align}\label{I}
I_1 \le C \( \intO |\nabla u |^N dx \)^{\frac{q}{N}}.
\end{align}
Since the potential function $|x|^{-N}(\log \frac{R}{|x|})^{-\beta}$ does not have any singularity in $\Omega \setminus \( B(\frac{R}{2}) \cup Q_\delta \)$, the Sobolev inequality yields that 
\begin{align}\label{II}
I_2 \le C \intO |u|^q dx \le C \( \intO | \nabla u|^N dx \)^{\frac{q}{N}}.
\end{align}
Finally, we shall derive a estimate of $I_3$ from above. 
Since $\log t \ge \frac{1}{2} (t-1) \,(1 \le t \le 2)$, we obtain
\begin{align}\label{III'}
I_3 \le C \int_{Q_\delta} \frac{|u(x)|^q}{(R-|x|)^{\beta}} dx \le C \int_{z_N = 0}^{z_N = \delta} \int_{z_N \ge C_1 |z'|^{\alpha}} \frac{|\tilde{u} (z', z_N)|^q}{|z|^{\beta}} dz,
\end{align}
where $u(x)=\tilde{u}(z) \,(z=x + (0, \cdots, 0, R))$.
If $\beta < \frac{N-1}{\alpha}+1$, then there exists $\ep >0$ and $p> \frac{N}{N-\ep}$ such that $(\beta -\ep )p < \frac{N-1}{\alpha} +1$.
By using the H\"older inequality and the Sobolev inequality, we have
\begin{align*}
&\int_{z_N = 0}^{z_N = \delta} \int_{z_N \ge C_1 |z'|^{\alpha}} \frac{|\tilde{u} (z', z_N)|^q}{|z|^{\beta}} dz \\
&= \iint \frac{|\tilde{u}|^\ep}{|z|^\ep} |\tilde{u}|^{q-\ep} |z|^{\beta -\ep} dz \\
&\le \( \iint \frac{|\tilde{u}|^N}{|z|^N} dz \)^{\frac{\ep}{N}} \( \iint |\tilde{u}|^{(q-\ep ) \frac{Np}{Np-N-p\ep }} dz \)^{\frac{Np-N-p\ep}{Np}} \( \iint |z|^{-(\beta -\ep )p} dz \)^{\frac{1}{p}} \\
&\le C \( \iint \frac{|\tilde{u}|^N}{|z_N|^N} dz \)^{\frac{\ep}{N}} \( \intO |\nabla \tilde{u}|^N dz \)^{\frac{q-\ep}{N}} \( \iint_{|z^{\prime}| \le \( \frac{ z_N }{C_1} \)^{\frac{1}{\alpha}}} z_N^{-(\beta -\ep )p} dz \)^{\frac{1}{p}} \\
&\le C \( \iint \frac{|\tilde{u}|^N}{|z_N|^N} dz \)^{\frac{\ep}{N}} \( \intO |\nabla \tilde{u}|^N dz \)^{\frac{q-\ep}{N}} \( \int_{z_N=0}^{z_N=\delta} z_N^{\frac{N-1}{\alpha} -(\beta -\ep )p} dz_N \)^{\frac{1}{p}}.
\end{align*}
Since $\frac{N-1}{\alpha} - (\beta -\ep )p > -1$, $\int_{0}^{\delta} z_N^{\frac{N-1}{\alpha} -(\beta -\ep )p} dz_N < \infty$.
Furthermore, applying the Hardy inequality on the half space $\re^N_{+}:= \{ \,(x', x_N) \in \re^{N-1} \times \re \,|\, x_N >0 \,\}$:
\begin{align*}
\( \frac{r-1}{r} \)^r \int_{\re^N_{+}} \frac{|u|^r}{|x_N |^r} dx \le \int_{\re^N_{+}} | \nabla u |^r dx  \quad (1 \le r < \infty )
\end{align*}
yields that
\begin{align}\label{III"}
\int_{z_N = 0}^{z_N = \delta} \int_{z_N \ge C_1 |z'|^{\alpha}} \frac{|\tilde{u} (z', z_N)|^q}{|z|^{\beta}} dz \le C \( \int_{\Omega + (0, \cdots, 0, R)} | \nabla \tilde{u} |^N dz \)^{\frac{q}{N}}.
\end{align}
By (\ref{III'}) and (\ref{III"}), we have
\begin{align}\label{III}
I_3 \le C  \( \intO | \nabla u|^N dx \)^{\frac{q}{N}}.
\end{align}
Therefore, from (\ref{123}), (\ref{I}), (\ref{II}), and (\ref{III}), for all $u \in W_0^{1,N}(\Omega)$,
\begin{align*}
C \( \intO \frac{|u(x)|^q}{|x|^N \( \log \frac{R}{|x|} \)^{\beta} } dx \)^{\frac{N}{q}} \le \intO | \nabla u|^N dx.
\end{align*}
Hence 
\begin{align}\label{G>0}
G_1 >0 \quad \text{at least for} \quad \beta < \frac{N-1}{\alpha}+1.
\end{align}
From (\ref{G=0}) and (\ref{G>0}), there exists $\beta^{*} \in [\frac{N-1}{\alpha} +1, \frac{N}{\alpha}]$ such that 
$G_1 > 0$ for $\beta < \beta^{*}$ and $G_1 = 0$ for $\beta > \beta^{*}$.

Lastly we shall show that $G_1$ is attained for $\beta \in (\frac{N-1}{N}q+1, \beta^{*})$.
In order to show it, we show that the continuous embedding $W_0^{1,N}(\Omega) \hookrightarrow L^q(\Omega; f_{1, \beta}(x) dx)$ is compact if $\frac{N-1}{N}q+1 < \beta <\beta^{*}$. 
Let $(u_m)_{m=1}^{\infty} \subset W_0^{1,N}(\Omega)$ be a bounded sequence. Then there exists a subsequence $(u_{m_k})_{k=1}^{\infty}$ such that
\begin{align}\label{L^r x}
&u_{m_k} \rightharpoonup u \,\, \text{in} \,\, W_0^{1,N}(\Omega ), \nonumber \\
&u_{m_k} \to u \,\, \text{in} \,\, L^r(\Omega ) \quad \text{for all} \,\, 1\le r < \infty.
\end{align}
We divide the domain into two parts as follows:
\begin{align}\label{12}
\intO \frac{|u_{m_k} - u|^q}{|x|^N \( \log \frac{R}{|x|} \)^{\beta} } dx 
= \int_{\Omega \setminus Q_\delta} + \int_{Q_\delta } =: J_1(u_{m_k} -u) + J_2(u_{m_k} -u).
\end{align}
Since $\log \frac{R}{|x|} \ge C \log \frac{aR}{|x|}$ for any $x \in \Omega \setminus Q_\delta$ for some $a>1$ and $C>0$, it holds that
\begin{align*}
J_1(u_{m_k} -u) \le C \int_{\Omega \setminus Q_\delta} \frac{|u_{m_k} - u|^q}{|x|^N \( \log \frac{aR}{|x|} \)^{\beta} } dx \le C \int_{\Omega} \frac{|u_{m_k} - u|^q}{|x|^N \( \log \frac{aR}{|x|} \)^{\beta} } dx.
\end{align*}
Note that the continuous embedding $W_0^{1,p}(\Omega) \hookrightarrow L^q(\Omega; f_{a,\beta}(x) dx)$ is compact for $\beta >\frac{N-1}{N}q+1$ from Lemma \ref{emb}, we obtain
\begin{align}\label{J_1}
J_1(u_{m_k} -u) \to 0 \quad \text{as}\,\, k \to \infty.
\end{align}
On the other hand, for any $\ep >0$, we take $\gamma >0$ which satisfies $\beta <\gamma <\beta^{*}$ and $(\log \frac{R}{|x|})^{\gamma -\beta} < \ep$ for $x \in Q_\delta$ (If necessary, we take small $\delta >0$ again.).
Then we have
\begin{align}\label{J_2}
J_2(u_{m_k} -u) &\le \ep \int_{Q_\delta} \frac{|u_{m_k} - u|^q}{|x|^N \( \log \frac{R}{|x|} \)^{\gamma} } dx \le C \ep \( \intO |\nabla (u_{m_k}-u) |^N \,dx \)^{\frac{q}{N}} \le C \ep.
\end{align}
From (\ref{12}), (\ref{J_1}), and (\ref{J_2}), we have
\begin{align*}
\intO \frac{|u_{m_k} - u|^q}{|x|^N \( \log \frac{R}{|x|} \)^{\beta} } dx \to 0 \quad \text{as}\,\, k \to \infty.
\end{align*}
Therefore the continuous embedding $W_0^{1,N}(\Omega) \hookrightarrow L^q(\Omega; f_{1,\beta}(x) dx)$ is compact if $\frac{N-1}{N}q+1 < \beta <\beta^{*}$. 
In conclusion, we have showed that $G_1$ is attained if $\frac{N-1}{N}q+1 < \beta <\beta^{*}$.
\end{proof}


Next we extend Theorem \ref{Thm ball} to general bounded domains.

\begin{theorem}\label{Thm general}
Let $a>1$. Then the followings hold.

\noindent
(i) If \,$\beta > \frac{N-1}{N}q +1$, then $G_a$ is attained for any bounded domains $\Omega$.

\noindent
(ii) If $\beta = \frac{N-1}{N}q +1, q > N$, and $a \ge e^{\frac{\beta}{N}}$, then $G_a=G_{\rm rad}$ and $G_a$ is not attained for any bounded domain $\Omega$.

\noindent
(iii) If $\beta = \frac{N-1}{N}q +1, q > N$, and $\Omega$ satisfies either $(\Omega_4)$ or $(\Omega_5)$, where

\noindent
$(\Omega_4):$ $\pd \Omega$ satisfies the Lipschitz condition at some point $x_0 \in \Omega \cap B_R(0)$,

\noindent
$(\Omega_5):$ $\Omega$ satisfies $(\Omega_1)-(\Omega_3)$ and $\alpha$ in $(\Omega_3)$ is grater than $\frac{N}{\beta}$,

\noindent
then there exists $a_* \in (1, e^{\frac{\beta}{N}}]$ such that $G_a$ is attained for $a \in (1,a_*)$ and $G_a$ is not attained for $a > a_*$.
\end{theorem}

In order to show Theorem \ref{Thm general} (iii), we need the continuity of $G_a$ with respect to $a$ at $a=1$. Under the assumptions $(\Omega_4)$, $(\Omega_5)$, we can show the continuity of $G_a$ at $a=1$ as follows.

\begin{lemma}\label{lemma cusp}
Let $\beta > N$. If $\Omega$ satisfies either $(\Omega_4)$ or $(\Omega_5)$, then $G_1 = \lim_{a\searrow 1} G_a =0$.
\end{lemma}

Lemma \ref{lemma cusp} follows from the following proposition.

\begin{prop}\label{prop cusp}
Let $a >1$. If $\Omega$ satisfies either $(\Omega_4)$ or $(\Omega_5)$, then there exists $C>0$ such that for $a$ close to $1$, $G_a \le C (a-1)^{\frac{N}{q}(\beta -\frac{N}{\alpha})}$, where $\alpha$ is regarded as $1$ if $\Omega$ satisfies $(\Omega_4)$.
\end{prop}

\begin{proof}[Proof of Proposition \ref{prop cusp}]
Let $x_a = R(2-a)\frac{x_0}{|x_0|}$ and $\phi \in C_c^{\infty}(B_1(0))$. Here $x_0$ is regarded as $(0, \cdots, 0,-R)$ if $\Omega$ satisfies $(\Omega_5)$. Then $B_{c(a-1)^{\frac{1}{\alpha}}}(x_a) \subset \Omega$ for $a$ close to $1$ and for sufficiently small $c>0$. Set $\phi_{a}(x)=\phi\( \frac{x-x_a}{c(a-1)^{\frac{1}{\alpha}}} \)$. 
Since $\log \frac{1}{t} \le \frac{1-t}{2}$ for $t$ close to $1$, we have the followings for $a$ close to $1$.
\begin{align*}
&\intO | \nabla \phi_a |^N \,dx = \int_{B_1(0)} |\nabla \phi |^N \,dz < \infty, \\
&\intO \frac{|\phi_a |^q}{|x|^N \( \log \frac{aR}{|x|} \)^{\beta} } dx 
\ge c^N (a-1)^{\frac{N}{\alpha}} \| \phi \|^q_{L^q} \( \log \frac{aR}{R(2-a) - c(a-1)^{\frac{1}{\alpha}}} \)^{-\beta} \ge C(a-1)^{\frac{N}{\alpha}-\beta}.
\end{align*}
\end{proof}

\begin{proof}[{\bf Proof of Theorem \ref{Thm general}}]
(i) We can check that Lemma \ref{emb} holds true for any bounded domains $\Omega$. Therefore (i) follows from the compactness of the embedding $W_0^{1,N}(\Omega) \hookrightarrow L^q(\Omega; f_{a, \, \beta}(x) dx)$. We omit the proof.

\noindent
(ii) Note that $W_{0,{\rm rad}}^{1,N}(B_\ep(0)) \subset W_{0}^{1,N}(\Omega)  \subset W_{0}^{1,N}(B_R(0))$ for small $\ep$ by zero extension. Then we have
\begin{align}\label{zero extension}
\inf_{u \in W_{0,{\rm rad}}^{1,N}(B_\ep (0) ) \setminus \{ 0\} } \dfrac{\intO | \nabla u |^N \,dx}{\( \intO \frac{|u|^q}{|x|^N (\log \frac{aR}{|x|})^{\beta}} dx \)^{\frac{N}{q}}}
\ge G_a 
&\ge \inf_{u \in W_0^{1,N}(B_R(0) ) \setminus \{ 0\} } \dfrac{\intO | \nabla u |^N \,dx}{\( \intO \frac{|u|^q}{|x|^N (\log \frac{aR}{|x|})^{\beta}} dx \)^{\frac{N}{q}}} \notag \\
&=\inf_{u \in W_{0, {\rm rad}}^{1,N}(B_R(0) ) \setminus \{ 0\} } \dfrac{\intO | \nabla u |^N \,dx}{\( \intO \frac{|u|^q}{|x|^N (\log \frac{aR}{|x|})^{\beta}} dx \)^{\frac{N}{q}}},
\end{align}
where the last equality comes from $a \ge e^{\frac{\beta}{N}}$. From the proof of Proposition \ref{G_rad} (ii), we can observe that $G_{\rm rad}$ does not vary even if we replace $W_{0,{\rm rad}}^{1,N}(B_R(0))$ to $W_{0,{\rm rad}}^{1,N}(B_\ep(0))$ for any small $\ep >0$. Thus the right hand side and the left hand side of (\ref{zero extension}) take same value, that is $G_{\rm rad}$. Therefore we have $G_a =G_{\rm rad}$. Furthermore if we assume that $G_a$ is attained by $u \in W_0^{1,N}(\Omega)$, then $u \in W_0^{1,N}(B_R(0))$ is also a minimizer on a ball. This contradicts Theorem \ref{Thm ball} (ii) in \S 2. Hence $G_a$ is not attained for any bounded domains $\Omega$.

\noindent
(iii) Note that $G_a$ is continuous with respect to $a \in (1,\infty)$, and is monotone increasing with respect to $a \in [1,\infty)$ for any bounded domains. 
From Lemma \ref{lemma cusp} and Theorem \ref{Thm general} (ii), we can show that there exists $a_* \in (1, e^{\frac{\beta}{N}}]$ such that $G_a < G_{\rm rad}$ for $a \in (1,a_*)$ and $G_a =G_{\rm rad}$ for $a > a_*$ in the same way as the proof of Theorem \ref{Thm ball} (ii). 
The remaining parts of the proof are similar to the proof of Theorem \ref{Thm ball} (ii).
\end{proof}

\section{Symmetry breaking}\label{sym break}

In this section, we consider radially symmetry of the minimizers of $G_a$ when $\Omega = B_R(0)$. 
We can show that any minimizer of $G_a$ has axial symmetry by using {\it spherical symmetric rearrangement}, see \cite{K}.  
Namely, for any minimizer $u_\beta$ of $G_a$ there exists some $\xi \in \mathbb{S}^{N-1}$ such that the restriction of $u_\beta$ to any sphere $\pd B_r(0)$ is symmetric decreasing with respect to the distance to $r \,\xi$. See also \cite{SW}.
The last result is as follows.

\begin{theorem}\label{Thm sym}
Let $\beta > \frac{N-1}{N} q +1, a>1$, and $u_\beta$ be a minimizer of $G_a$ in Theorem \ref{Thm ball} (i). Then the followings hold true.

\noindent
(i) For fixed $q > N$, there exists $\beta_*$ such that $u_\beta$ is non-radial for $\beta > \beta_*$. 

\noindent
(ii) $u_\beta$ is radial for any $\beta$ and $q \le N$.
\end{theorem}

In order to show Theorem \ref{Thm sym} (i), we need two lemmas concerning growth orders of $G_a$ and $G_{a, {\rm rad}}$ with respect to $\beta$.

\begin{lemma}\label{Lemma G}
For fixed $q > N$, there exists $C>0$ such that for sufficiently large $\beta$ the following estimate holds true.
\begin{align*}
G_a \le C \beta^{\frac{N^2}{q}}(\log a)^{\frac{N \beta}{q}}.
\end{align*}
\end{lemma}

\begin{proof}[{\bf Proof of Lemma \ref{Lemma G}}]
Let $u \in C_c^{\infty}(B_R(0))$. Following \cite{SSW} we consider $u_\beta (x):=u( \beta (x-x_\beta ) )$ for $x \in B_{\beta^{-1}}(x_\beta)$, where $x_\beta :=(R-\beta^{-1}, 0, \cdots, 0) \in B_R(0)$. Then for sufficiently large $\beta$ we obtain 
\begin{align}\label{nume}
\int_{B_{\beta^{-1}}(x_\beta)} |\nabla u_\beta (x)|^N dx &= \int_{B_R(0)} |\nabla u (y)|^N dy, \\
\label{deno}
\int_{B_{\beta^{-1}}(x_\beta )} \dfrac{|u_\beta (x)|^q}{|x|^N (\log \frac{aR}{|x|})^\beta} dx &\ge 
\( R-2\beta^{-1} \)^{-N} \( \log \frac{aR}{R-2\beta^{-1}} \)^{-\beta} \beta^{-N} \int_{B_R(0)} |u(y)|^q dy.
\end{align}
We set $f(\beta):=( R-2\beta^{-1} )^{-N} ( \log \frac{aR}{R-2\beta^{-1}} )^{-\beta}$. Since $\log \frac{1}{1-x} \le 2x$ for all $x \in [0,\frac{1}{2}]$, for large $\beta$ we have
\begin{align*}
f(\beta ) &\ge \frac{1}{2} \( \log a + \log \frac{1}{1-2 \beta^{-1} R^{-1}} \)^{-\beta} \\
&\ge \frac{1}{2} \(\log a + 4\beta^{-1} R^{-1} \)^{-\beta} \\
&= \frac{1}{2} (\log a )^{-\beta} \( 1+ \frac{4}{\beta R\log a} \)^{-\beta} 
\end{align*}
which yields that
\begin{align}\label{f beta}
f(\beta ) \ge C (\log a )^{-\beta} \quad \text{for large} \,\, \beta.
\end{align}
From (\ref{nume}), (\ref{deno}), and (\ref{f beta}), we obtain
\begin{align*}
G_a \le \dfrac{\int_{B_R(0)} |\nabla u_\beta |^N dx}{ \( \int_{B_R(0)} \frac{|u_\beta |^q}{|x|^N (\log \frac{aR}{|x|})^\beta} dx\)^{\frac{N}{q}}} \le C \beta^{\frac{N^2}{q}}(\log a)^{\frac{N\beta}{q}}.
\end{align*}
\end{proof}

\begin{lemma}\label{Lemma GR}
For fixed $q > N$, there exists $C>0$ such that for sufficiently large $\beta$ the following estimate holds true.
\begin{align*}
G_{a, {\rm rad}} \ge C \beta^{N-1+\frac{N}{q}} \( \log a \)^{\frac{N\beta}{q} -(N-1+\frac{N}{q})}.
\end{align*}
\end{lemma}

\begin{proof}[{\bf Proof of Lemma \ref{Lemma GR}}] 
For $u \in W_{0,\text{rad}}^{1,N}(B_R(0))$ we define $v \in W_{0,\text{rad}}^{1,N}(B_R(0))$ as follows:
\begin{align*}
v(s) = u(r), \,\, \text{where}\,\,  ( \log a )^{A-1} \log \frac{aR}{s} = \( \log \frac{aR}{r} \)^A \,\text{and}\, A=\frac{N(\beta -1)}{(N-1)q}.
\end{align*}
Direct calculation shows that
\begin{align*}
\int_{B_R(0)} \dfrac{|u |^q}{|x|^N (\log \frac{aR}{|x|})^\beta} dx 
&=\w_{N-1} \int_0^R |u(r)|^q \( \log \frac{aR}{r} \)^{-\beta} \frac{dr}{r} \\
&=\w_{N-1} A^{-1} (\log a)^{\frac{A-1}{A} (1-\beta )} \int_0^R \dfrac{|v(s)|^q}{s (\log \frac{aR}{s})^{\frac{A-1+\beta}{A}}} ds \\
&= A^{-1} (\log a)^{\frac{A-1}{A} (1-\beta )} \int_{B_R(0)} \dfrac{|v |^q}{|y|^N (\log \frac{aR}{|y|})^{\frac{N-1}{N}q +1}} dy.
\end{align*}
In the same way as above, we have
\begin{align*}
\int_{B_R(0)} |\nabla u|^N dx &= A^{N-1} ( \log a)^{-\frac{A-1}{A}} \int_{B_R(0)} |\nabla v |^N  \( \log \frac{aR}{|y|}\)^{\frac{A-1}{A}} dy \\
&\ge A^{N-1} \int_{B_R(0)} |\nabla v |^N dy.
\end{align*}
Therefore we have 
\begin{align*}
\dfrac{\int_{B_R(0)} | \nabla u |^N \,dx}{\( \int_{B_R(0)} \frac{|u|^q}{|x|^N (\log \frac{aR}{|x|})^{\beta}} dx \)^{\frac{N}{q}}} 
\ge A^{N-1+\frac{N}{q}} (\log a)^{\frac{N}{q}(\beta -1) \frac{A-1}{A}} \dfrac{\int_{B_R(0)} | \nabla v |^N \,dy}{\( \int_{B_R(0)} \frac{|v|^q}{|y|^N (\log \frac{aR}{|y|})^{\frac{N-1}{N}q+1}} dy \)^{\frac{N}{q}}}
\end{align*}
which yields that 
\begin{align*}
G_{a, {\rm rad}} \ge \( \frac{N(\beta -1)}{(N-1) q} \)^{N-1+\frac{N}{q}} (\log a)^{\frac{N\beta}{q} -\( N-1+\frac{N}{q} \)} \inf_v \dfrac{\int_{B_R(0)} | \nabla v |^N \,dy}{\( \int_{B_R(0)} \frac{|v|^q}{|y|^N (\log \frac{aR}{|y|})^{\frac{N-1}{N}q+1}} dy \)^{\frac{N}{q}}}.
\end{align*}
Therefore, for sufficiently large $\beta$ we have 
\begin{align*}
G_{a, {\rm rad}} \ge C \beta^{N-1+\frac{N}{q}} \( \log a \)^{\frac{N\beta}{q} -(N-1+\frac{N}{q})}.
\end{align*}
\end{proof}

Finally we shall show Theorem \ref{Thm sym}.

\begin{proof}[{\bf Proof of Theorem \ref{Thm sym}}] 
(i) It is enough to show that $G_a < G_{a,{\rm rad}}$. 
By Lemma \ref{Lemma G} and Lemma \ref{Lemma GR}, for fixed $q >N$ there exists $\beta_*$ such that for $\beta > \beta_*$ 
\begin{align*}
G_a \le C \beta^{\frac{N^2}{q}}(\log a)^{\frac{N \beta}{q}} < C \beta^{N-1+\frac{N}{q}} (\log a)^{\frac{N\beta}{q} - (N-1+\frac{N}{q})} \le G_{a, {\rm rad}},
\end{align*}
since $\frac{N^2}{q} < N-1+\frac{N}{q}$. Therefore we see that $G_a < G_{a,{\rm rad}}$.

\noindent
(ii) Let $x=r \w \,(r=|x|, \w \in S^{N-1})$ for $x \in B_R(0)$. For $u \in W_0^{1,N}(B_R(0))$ we consider the following radial function $U$:
\begin{align*}
U(r)=\(\w_{N-1}^{-1} \int_{S^{N-1}} |u(r \w)|^{N} dS_{\w} \)^{\frac{1}{N}}.
\end{align*} 
Then we have
\begin{align*}
	U^{\,\prime}(r) \le \( \w_{N-1}^{-1} \int_{S^{N-1}} \left| \frac{\pd}{\pd r} u(r \w ) \right|^N dS_{\w} \)^{\frac{1}{N}}
\end{align*}
which yields that
\begin{equation}\label{rad non-rad}
	\int_{B_R(0)} | \nabla U |^N dx \le \int_{B_R(0)} \left| \nabla u \cdot \frac{x}{|x|} \right|^N dx \le \int_{B_R(0)} | \nabla u |^N dx.
\end{equation}
On the other hand, we have
\begin{align}
\int_{B_R(0)} \dfrac{|U|^q}{|x|^N (\log \frac{aR}{|x|})^{\beta}} dx 
&= \w_{N-1} \int_0^R \( \w_{N-1}^{-1} \int_{S^{N-1}} |u(r\w )|^N dS_{\w} \)^{\frac{q}{N}} \dfrac{dr}{r (\log \frac{aR}{|x|})^\beta} \nonumber \\
&\ge \int_0^R \int_{S^{N-1}} |u(r\w )|^q dS_{\w}  \dfrac{dr}{r (\log \frac{aR}{|x|})^\beta} \nonumber \\
\label{rad non-rad 2}
&=\int_{B_R(0)} \frac{|u|^q}{|x|^N (\log \frac{aR}{|x|})^{\beta}} dx
\end{align}
where the inequality follows from Jensen's inequality and $q \le N$.
From (\ref{rad non-rad}) and (\ref{rad non-rad 2}), we obtain $G_{a,{\rm rad}} \le G_a$. 
Therefore $G_{a, {\rm rad}} =G_a$ for any $q \le N$ and $\beta$. 
Moreover we observe that any minimizers of $G_a$ must be radial from the equality condition of (\ref{rad non-rad}).
\end{proof}


\section*{Acknowledgment}
Part of this work was supported by JSPS Grant-in-Aid for Fellows (PD), No. 18J01053. 
The author would like to thank Professor Marta Calanchi (The University of Milan) for valuable advice and encouragement.


\end{document}